\documentclass[11pt]{amsart}
\usepackage{amsmath,amsfonts,latexsym,amssymb,amsthm, verbatim}
\usepackage{url}
 \usepackage[OT2, T1]{fontenc}

\newtheorem{tm}{Theorem}
\newtheorem{proposition}[tm]{Proposition}
\newtheorem{lemma}[tm]{Lemma}
\newtheorem{corollary}[tm]{Corollary}

\theoremstyle{definition}
\newtheorem{remark}[tm]{Remark}

\newcommand{\Q}{\mathbb{Q}}

\newcommand{\Z}{\mathbb{Z}}

\newcommand{\F}{\mathbb{F}}

\DeclareMathOperator{\Aut}{Aut}

\DeclareMathOperator{\Gal}{Gal}
\DeclareMathOperator{\GL}{GL}

\DeclareMathOperator{\SL}{SL}

\def\diam#1{\langle#1\rangle}

\begin{document}

\title[Torsion of elliptic curves over cubic fields and points on $X_1(n)$]{Torsion of rational elliptic curves over cubic fields and sporadic points on $X_1(n)$}
\author{Filip Najman}
\address{Department of Mathematics\\ University of Zagreb\\ Bijeni\v cka cesta 30\\ 10000 Zagreb\\ Croatia}
\email{fnajman@math.hr}
\thanks{The author was supported by the Ministry of Science, Education, and Sports, Republic of Croatia, grant 037-0372781-2821.}
\begin{abstract}
We classify the possible torsion structures of rational elliptic curves over cubic fields. Along the way we find a previously unknown torsion structure over a cubic field, $\Z /21 \Z$, which corresponds to a sporadic point on $X_1(21)$ of degree 3, which is the lowest possible degree of a sporadic point on a modular curve $X_1(n)$.
\end{abstract}
\maketitle

\section{Introduction}

When trying to understand elliptic curves over number fields, an important problem is to classify the possible torsion structures. The first such classification was done by Mazur \cite{maz1,maz2}, proving that the torsion of an elliptic curve over $\Q$ has to be isomorphic to one of the following 15 groups:
\begin{equation}
\vcenter{\openup\jot\halign{$\hfil#\hfil$\cr
\Z/m \Z, 1 \leq m\leq 12,\ m\neq 11,\cr
\Z/2 \Z \oplus \Z/2m \Z,\ 1 \leq m\leq 4.\cr}}
\label{rac}
\end{equation}
After this result, attention shifted toward number fields. Kamienny \cite[Theorem 3.1]{Kam1} proved that if a torsion point of an elliptic curve over a quadratic field has prime order $p$, then $p\leq 13$. This, when combined with a theorem of Kenku and Momose \cite[Theorem (0.1).]{km}, gave a complete list of possible torsion structures over quadratic fields:
\begin{equation}
\vcenter{\openup\jot\halign{$\hfil#\hfil$\cr
\Z/m \Z, 1 \leq m\leq 18,\  m\neq 17,\cr
\Z/2 \Z \oplus \Z/2m \Z,\ 1 \leq m\leq 6,\cr
\Z/3 \Z \oplus \Z/3m \Z,\ m=1,2,\cr
\Z/4 \Z \oplus \Z/4 \Z.\cr}}
\label{kvad}
\end{equation}
The author gave a similar complete list for the fields $\Q(i)$ and $\Q(\sqrt{-3})$ (see \cite[Theorem 2.]{naj0}), and a procedure how to make such a list was developed by Kamienny and the author \cite{kn}.

As one can see, much is known about the possible torsion structures of elliptic curves over $\Q$ and over quadratic fields. Unfortunately, already over cubic fields a classification of possible torsion structures of elliptic curves is not known. However, it is known that if an elliptic curve over a cubic field has a point of prime order $p$, then $p\leq 13$ (see \cite{Par1, Par2}).  Jeon, Kim and Schweizer \cite[Theorem 3.4.]{jks} found all the torsion structures that appear infinitely often as one runs through all elliptic curves over all cubic fields:
\begin{equation}
\vcenter{\openup\jot\halign{$\hfil#\hfil$\cr
\Z/m \Z, 1 \leq m\leq 20,\ m\neq 17,19,\cr
\Z/2 \Z \oplus \Z/2m \Z,\ 1 \leq m\leq 7.\cr}}
\label{kub}
\end{equation}
Jeon, Kim and Lee \cite{jkl} constructed infinite families having each of the torsion structures from the list (\ref{kub}).

However, it was unknown whether the list (\ref{kub}) is complete, i.e. if one runs through all elliptic curves over all cubic fields, do there exist torsion structures that appear only finitely many times? We show, by finding an elliptic curve with torsion $\Z /21 \Z$ over a cubic field, that the answer is yes and that the list (\ref{kub}) is not the complete list of possible torsion structures over cubic fields. Note that this, in contrast with what happens over $\Q$ and over quadratic fields, where each group that can appear at all, is the torsion group of infinitely many non-isomorphic elliptic curves.

The main purpose of this paper is to study the possible torsion structures of all rational elliptic curves (meaning that all their coefficients are $\Q$-rational) over all cubic fields. This is a natural question to consider as, apart from being interesting in itself, it is often important to study rational elliptic curves over extensions of $\Q$ when solving Diophantine equations (see for example \cite{bru}). Somewhat similar problems were studied by Fujita \cite{fuj2}, who studied the possible torsion groups of rational elliptic curves over the compositum of all quadratic fields, by Lozano-Robledo \cite{alr}, who studied the minimal degree of the field of definition of points of order $p$ on rational elliptic curves and by Gonz\'alez-Jim\'enez and Tornero \cite{gt}, who studied how can the torsion of a rational elliptic curve grow upon base changing to a quadratic field.

The main result of this paper is the following theorem.
\begin{tm}
Let $E/\Q$ be a rational elliptic curve, and let $K/ \Q$ be a cubic extension. Then $E(K)_{tors}$ is one of the following groups
\label{gltm}
\begin{equation}
\vcenter{\openup\jot\halign{$\hfil#\hfil$\cr
\Z/m \Z, \text{ }m=1,\ldots,10, 12,13,14,18,21,\cr
\Z/2 \Z \oplus \Z/2m \Z, \text{ } m=1,2,3,4,7.\cr}}
\label{lista}
\end{equation}
The elliptic curve $162B1$ over $\Q(\zeta_9)^+$ is the unique rational elliptic curve over a cubic field with torsion $\Z/ 21\Z$.
For all the other groups $T$ in the list \eqref{lista}, there exists infinitely many rational elliptic curves that have torsion $T$ over some cubic field.
\end{tm}

To prove Theorem \ref{gltm}, we will first need to solve the analogous problem for quadratic fields, which we do in Section \ref{sec:kvad}.

We prove Theorem \ref{gltm} by studying the action of the Galois group on the torsion points, division polynomials, and by finding the rational points on certain (modular) curves.

Let us mention that a somewhat similar problem to the one considered in this paper is the problem of finding the possible torsion structures of rational elliptic curves with integral $j$-invariant \cite{msz,pwz} and with complex multiplication over number fields \cite{diu,cla}.

Recall that the $\emph{gonality}$ $\gamma(X)$ of an algebraic curve $X$ is the lowest degree of a nonconstant rational map from $X$ to the projective line. We call points of degree $d$ on the modular curves $Y_1(m,n)$ (see Section 2 for definitions of modular curves and note that we only consider modular curves with $m=1,2$ in this paper), when $d<\gamma(Y_1(m,n))$ \emph{sporadic}. Since all the modular curves $Y_1(m,n)$ that correspond to the torsion structures in the list(\ref{rac}) are of genus 0 and have (infinitely many) rational points (since some of the cusps of $X_1(m,n)$ are rational) and hence are of gonality 1. Similarly, all the modular curves $Y_1(m,n)$ that correspond to the torsion structures in the list (\ref{kvad}) are of genus $\leq 2$ (and hence have gonality 1 or 2), so it follows that there are no sporadic points of degree 1 or 2. Van Hoeij \cite{hoe} found sporadic points of degree $6$ on $X_1(37)$ (of gonality 18), and of degree 9 on $X_1(29)$ and $X_1(31)$ (of gonality 11 and 12, respectively).

Since the modular curve $X_1(21)$ has gonality $4$, the unique rational elliptic curve with $21$-torsion over a cubic field gives us a degree 3 sporadic point, which is the lowest degree possible. By the method used to construct this point we rediscover van Hoeij's degree 6 point on $X_1(37)$.

\section{Conventions and notation}

Throughout this paper, $K$ will be a cubic field and $L$ will be its normal closure over $\Q$. This means that when $K/ \Q$ is normal, then $L=K$, and otherwise $L$ is a degree 6 extension of $\Q$ such that $\Gal(L/ \Q)\simeq S_3$. We denote by $M$ the unique field with the property that $M$ is a subfield of $L$ such that $[L:M]=3$ (from which it follows that $\Gal(L/ M)\simeq \Z / 3 \Z$). If $K$ is normal over $\Q$, this will mean that $M= \Q$.

Let $E[n]=\{ P\in E(\overline \Q )| nP=0 \}$ denote the $n$-th division group of $E$ over $\overline \Q$ and let $\Q(E[n])$ be the $n$-th division field of $E$.

We will denote by $E^d$ a quadratic twist of $E$ by $d \in \Q^* / (\Q^*)^2$. By $\zeta_n$ we will denote a $n$th primitive root of unity and by $\Q(\zeta_n)^+$ the maximal real subfield of $\Q(\zeta_n)$.

For $n$ an odd positive integer, we denote by $\psi_n$ the $n$-th division polynomial of an elliptic curve $E$ (see \cite[Section 3.2]{was} for details), which satisfies that, for a point $P\in E$, $\psi_n(x(P))=0$ if and only if $nP=0$. As before, although the division polynomial depends on the curve $E$, we will leave $E$ out of the index as it will be clear what elliptic curve we are referring to.

Let $E/ \Q$ be an elliptic curve and $d>1$ be an integer.  Factor $\psi_n$ and then consider all factors of degree $l$, where $l|d$. Each of these factors generates a field $F$ over which the $x$-coordinate of a point $P$ such that $nP=0$ is defined. The torsion point $P$ is then defined over $F'$ which is either $F$ itself or a quadratic extension of $F$ obtained by adjoining the $y$-coordinate of $P$. By considering all such fields $F'$ of degree dividing $d$, we can check whether a fixed rational elliptic curve $E$ obtains $n$-torsion over some extension of degree $d$. We call this method the \emph{division polynomial method}.

The division polynomial method can also be effectively used in determining which, if any, twists of a given curve $E/F$ have non-trivial $n$-torsion over the base field. This is done by finding all the quadratic extensions $F(\sqrt d)$ in which the $n$-torsion grows and then computing for which of the finitely many $d$ obtained does the curve $E^d(F)$ have non-trivial $n$-torsion.

If there exists a $K$-rational cyclic isogeny $\phi:E\rightarrow E'$ of degree $n$, this implies that $\ker \phi$ is $\Gal(\overline K / K)$-invariant cyclic group of order $n$ and we will say that $E/K$ has an $n$-isogeny.

When counting rational elliptic curves, unless stated otherwise, we will count up to $\Q$-isomorphism. When referring to specific elliptic curves we will list them as they appear in Cremona's tables \cite{cre}, as we already did in Theorem \ref{gltm}.

Let $m|n$ and denote by $Y_1(m,n)$ the affine modular curve whose $K$-rational points classify isomorphism classes of the triples $(E, P_m, P_n)$, where $E$ is an elliptic curve (over $K$) and $P_m$ and $P_n$ are torsion points (over $K$) which generate a subgroup isomorphic to $\Z/m \Z \oplus \Z/n \Z$. For simplicity, we write $Y_1(n)$ instead of $Y_1(1,n)$. Let $X_1(m,n)$ be the compactification of the curve $Y_1(m,n)$ obtained by adjoining its cusps.

Denote by $Y_0(n)$ the affine curve whose $K$-rational points classify isomorphism classes of pairs $(E,C)$, where $E/K$ is an elliptic curve and $C$ is a cyclic (Gal$(\overline K/K)$-invariant) subgroup of $E$. Let $X_0(n)$ be the compactification of $Y_0(n)$.

Recall that a \emph{$\Q$-curve} is an elliptic curve $E/K$ over a number field, such that it is $\overline \Q$-isogenuos to all of its $\Gal(\overline \Q/ \Q)$-conjugates.

For computations we use Magma \cite{mag}.

\section{Torsion of rational elliptic curves over quadratic fields}
\label{sec:kvad}
The results of this short section will be needed in the proof of Theorem \ref{gltm}, but are also interesting in their own right. They also provide a nice introductory exercise for the much harder cubic fields case.

\begin{tm}
Let $E$ be a rational elliptic curve and $F$ a quadratic field.
\begin{itemize}
\item[a)] The torsion of $E(F)$ is isomorphic to one of the following groups
\begin{equation}
\vcenter{\openup\jot\halign{$\hfil#\hfil$\cr
\Z/m \Z,\text{ }m=1,\ldots ,10,12,15,16\cr
\Z/2 \Z \oplus \Z/2m \Z, \text{ } 1\leq m \leq 6.\cr
\Z/3 \Z \oplus \Z/3m \Z, \text{ }  m=1,2,\cr
\Z/4 \Z \oplus \Z/4 \Z.\cr}}
\label{kvadlista}
\end{equation}
\item[b)] Each of these groups, except for $\Z/15 \Z$, appears as the torsion structure over a quadratic field for infinitely many rational elliptic curves $E$.
\item[c)] The elliptic curves $50B1$ and $50A3$ have $15$-torsion over $\Q(\sqrt 5)$, $50B2$ and $450B4$ have $15$-torsion over $\Q(\sqrt{-15} )$. These are the only rational curves having non-trivial $15$-torsion over any quadratic field.
\end{itemize}
\label{kvkl}
\end{tm}

\begin{remark}
\begin{itemize}
\item[a)] The elliptic curves 50B1 and 50A3 are twists by 5 of each other, and hence become isomorphic over $\Q(\sqrt 5)$. Similarly, 50B2 and 450B4 are twists of each other by $-15$ and become isomorphic over $\Q(\sqrt{-15})$.
\item[b)] These elliptic curves are "exceptional curves", in the sense that they are the only elliptic curves (not just rational) over the respective quadratic fields with $15$-torsion (see \cite{kn,naj2} for details).

\end{itemize}
\end{remark}

Before we proceed with the proof of Theorem \ref{kvkl}, we will prove the following lemma.
\begin{lemma}
Let $E$ be a rational elliptic curve. Then in the family of all quadratic twists $E^d$ of $E$ (including $E$ itself) there is at most one elliptic curve with nontrivial $n$-torsion for $n=5,7$, and at most $2$ curves with $3$-torsion.
\end{lemma}
\begin{proof}
We will often use the fact (see for example \cite[Theorem 3 and Corollary 4]{gt}) that if $F=\Q (\sqrt d)$, and $n$ an odd integer $>1$ then
\begin{equation}
E(F)[n]=E(\Q)[n]\oplus E^d(\Q)[n].
\label{zbrojkvad}
\end{equation}


Suppose that two twists $E^d$ and $E^{d'}$ of $E$ have non-trivial $n$-torsion, for $n=5$ or $7$.  Then $E^{d'}$ is a twist of $E^{d}$ by $d'/d$ and now \eqref{zbrojkvad} implies that $E^d(\Q(\sqrt{d/d'}))$ has full $n$-torsion, which is impossible.

It is impossible that $E$ has 3 twists with $3$-torsion, as this would imply that, by \cite[Lemma 9]{fuj2}, $E^d(F_2)$ would contain $(\Z /3\Z)^3$ for some twist $E^d$ of $E$ and some biquadratic extension $F_2$ of $\Q$.
\end{proof}

\begin{proof}[Proof of Theorem \ref{kvkl}]
The possible torsion structures of a rational elliptic curve over a quadratic field is obviously a subset of the list (\ref{kvad}). The equality (\ref{zbrojkvad}) rules out the possibility of $n$-torsion for $n=11,13$.
Note that the number of points of order $2$ over $F$ on an elliptic curve
$$E:y^2=f(x)=x^3+ax+b$$
is equal to the number of roots of $f$ over $F$. It follows that if $E(\Q)$ has no $2$-torsion, then neither does $E(F)$.

 Suppose $E(F)$ had $2n$-torsion, for $n=7$ or $9$. Then, as noted above, both $E(\Q)$ and $E^d(\Q)$ would have a $2$-torsion point, and by \eqref{zbrojkvad} it would follow that either $E(\Q)$ or $E^d(\Q)$ has a point of order $n$ and therefore also a point of order $2n$, which is impossible by Mazur's theorem.

It can be seen from \cite[Theorems 3.2., 3.3., 3.4., 3.5. and 3.6.]{jkl4} that there exist infinitely many rational elliptic curves with torsion $\Z/2 \Z \oplus \Z/10 \Z,\ \Z/2 \Z \oplus \Z/12 \Z,\ $ $\Z/3 \Z \oplus \Z/3 \Z,\ \Z/3 \Z \oplus \Z/6 \Z$  and $
\Z/4 \Z \oplus \Z/4 \Z$.

In \cite[Remark 2.6. (d)]{lol} one is given a construction which can be used to construct infinitely many rational elliptic curves with $16$-torsion over quadratic fields.

Finally, we wish to find all rational elliptic curves with $15$-torsion over quadratic fields. Let $E/ \Q$ be an elliptic curve which attains $15$-torsion over a quadratic field. By (\ref{zbrojkvad}) this implies that,
\begin{equation}
E(\Q)_{tors} \simeq \Z /3 \Z \text{ and }E^d(\Q)_{tors} \simeq \Z /5 \Z
\label{15-torz}
\end{equation} or vice versa. Suppose without loss of generality that it is as in (\ref{15-torz}). It also follows, since if $E$ has a $p$-isogeny, so do all the quadratic twists $E^d$ of $E$, that $E$ has to have a $15$-isogeny. But there are only 4 families of twists of rational elliptic curves with a $15$-isogeny \cite[p.78--80]{mod}, those with $j$-invariants
$$-25/2,\ -349938025/8,\ -121945/32,\  46969655/32768,$$
which are the twists of the elliptic curves 50A1, 50A2, 50B1 and 50B2, respectively.

 By the division polynomial method we find that 50B1 has $15$-torsion only over one quadratic field, namely $\Q(\sqrt 5)$, that 50B2 has $15$-torsion only over one quadratic field, namely $\Q(\sqrt{-15})$, and that 50A1 and 50A2 have no twists with $5$-torsion, completing the proof of the theorem.
\end{proof}

\section{Auxiliary results}
In this section we prove a series of results that we will need for the proof of Theorem \ref{gltm}.

\begin{lemma}
\label{spust}
Let $F/ \Q$ be a quadratic extension, $n$ an odd positive integer, and $E/ \Q$ an elliptic curve such that $E(F)$ contains $\Z /n\Z$. Then $E/ \Q$ has an $n$-isogeny.
\end{lemma}
\begin{proof}
Factor $n$ as $n=\prod_{i=1}^kp_i^{e_i}$, where $p_i$ and $p_j$ are distinct primes for $i\neq j$.
As already mentioned in (\ref{zbrojkvad}), if $F=\Q (\sqrt d)$, then for each $i=1,\ldots, k$
$$E(F)[p_i^{e_i}]=E(\Q)[p_i^{e_i}]\oplus E^d(\Q)[p_i^{e_i}].$$
Thus either $E(\Q)$ or $E^d(\Q)$ has a point of order $p_i^{e_i}$ and thus a $\Gal(\overline \Q /\Q)$-invariant subgroup generated by this point.  If an elliptic curve has a $\Gal(\overline \Q /\Q)$-invariant cyclic subgroup of order $p_i^{e_i}$, then so does every twist of $E$ and hence we conclude that $E$ has a $\Gal(\overline \Q /\Q)$-invariant cyclic subgroup of order $n$, proving the lemma.
\end{proof}

We will also extensively use the well-known classification of possible degrees of cyclic isogenies over $\Q$.

\begin{tm}[\cite{maz2,ken2,ken3,ken4}]
\label{izogklas}
Let $E/ \Q$  be an elliptic curve with an $n$-isogeny. Then $n\leq 19$ or $n\in \{21,25,27,37,43,67,163\}$. If $E$ does not have complex multiplication, then $n\leq 18$ or $n\in \{21,25,37\}$.
\end{tm}

Next we show that from 2 independent isogenies of degrees $m$ and $n$ on an elliptic curve $E$, one can deduce the existence of a $mn$-isogeny on an isogenous curve $E'$.

\begin{lemma}
\label{izogzbrajanje}
Let $E/F$ be an elliptic curve over a number field with $2$ independent isogenies (the intersection of their kernels is trivial) of degrees $m$ and $n$ (over $F$). Then $E$ is isogenous (over $F$) to an elliptic curve $E'/F$ with an $mn$-isogeny.
\end{lemma}
\begin{proof}
Suppose $E$ has an $m$-isogeny $f:E\rightarrow E'$ and an $n$-isogeny $g:E \rightarrow E''$. We claim that $h=g \circ \hat f:E'\rightarrow E'' $, where $\hat f$ is the dual isogeny of $f$, is a cyclic $mn$-isogeny.

Suppose the opposite, that $h$ is not cyclic. Then  for some integer $l\geq 2$ which divides both $m$ and $n$, $E'[l]\subset \ker h$. Note that
$$\hat f(E'[l])\simeq \Z/l\Z \text{ and }\hat f (E'[l]) \subset \ker f$$
since $f\circ \hat f=[n]_{E'}$. But since $\hat f(E'[l])$ contains non-zero points and $\ker g$ and $\ker f$ have trivial intersection, this means that $g$ cannot send $\hat f(E'[l])$ to $0$. Therefore $h(E'[l])\neq 0$, which is a contradiction.
\end{proof}

The following four results will be useful in controlling the 2-primary torsion of rational elliptic curves over cubic fields.

\begin{lemma}[{\cite[Lemma 1]{naj2}}]
\label{2-syl}
 If $E(\Q )$ has a nontrivial $2$-Sylow subgroup, $E(K)$ has the same $2$-Sylow subgroup as $E(\Q)$.
\end{lemma}

\begin{proposition}
\label{2-stvaranje}
Let $M \neq \Q(i)$, and let $E/ \Q$ be an elliptic curve such that $E(M)[2]=0$. Then the $2$-Sylow subgroup of $E(L)$ is either trivial or $\Z / 2\Z \oplus \Z / 2\Z$.
\end{proposition}

\begin{proof}
Suppose $E(L)[2]\neq 0 $. Write
$$E:y^2=f(x)=x^3+ax+b,$$
where $f(x)$ is irreducible over $M$. As $L/M$ is Galois, it follows that since $f$ has one root over $L$, all the roots of $f$ are defined over $L$. Hence $E(L)[2]\simeq \Z / 2\Z \oplus \Z / 2\Z$.

Suppose that $E(L)$ has a point of order 4. As $L$ does not contain $i$, it follows that the only possibility for $E(L)$ to have a $4$-torsion point is that $E(L)[4]\simeq \Z /2\Z \oplus \Z /4\Z$.

We now prove that a group of order 3 has to fix a line of $\Z /2\Z \oplus \Z /4\Z$. Let $G:=\Gal(L/M)$. We have the short exact sequence
\begin{equation}
\label{ses}
0\rightarrow E(L)[2] \rightarrow E(L)[4] \rightarrow E(L)[4]/E(L)[2] \rightarrow 0
\end{equation}
and it follows that
\begin{equation}
\label{ses2}
0 \rightarrow E(L)[2]^G \rightarrow E(L)[4]^G \rightarrow (E(L)[4]/E(L)[2])^G \rightarrow H^1(G,E(L)[2])
\end{equation}
It is easy to compute that $H^1(G,E(L)[2])=0$, and since $E(L)[4]/E(L)[2]$ is a group of order $2$, it follows that
$$(E(L)[4]/E(L)[2])^G\simeq \Z/2\Z,$$
from which we conclude that $E(L)[4]^G\neq 0$. We conclude that $E(M)$ has a $2$-torsion point, which is a contradiction.
\end{proof}

\begin{proposition}
\label{2-stvaranje2}
Let $M= \Q(i)$. Let $E/ \Q$ be an elliptic curve with no $\Q$-rational points of order $2$. Then
\begin{itemize}
\item[a)] If $E(K)$ has a point of order $4$, then $\Delta(E)\in -1\cdot (\Q^*)^2,$ $j(E)=-4t^3(t+8)$ for some $t\in \Q$ and $E(L)[4]\simeq \Z /4\Z \oplus \Z /4\Z$.

\item[b)] $E(K)$ has no points of order $8$.
\end{itemize}
\end{proposition}
\begin{proof}
First note that from the assumption that $E(\Q)[2]=0$, it follows that $E(\Q(i))[2]=0$. Suppose $E(L)[2]\neq0$. Since $L/ \Q (i)$ is Galois, it follows that $E(L)[2]\simeq\Z / 2\Z \oplus \Z / 2\Z$.

If $E(L)$ has a point of order 4, by the same argument as in the proof of Proposition \ref{2-stvaranje} it follows that $E(L)[4]$ cannot be $\Z/2\Z \oplus \Z/4\Z$. Thus
$$E(L)[4]\simeq \Z/4 \Z \oplus \Z/4\Z,$$
 and $\Q(E[2])=\Q(E[4]).$ Note that for any elliptic curve $E'$, the field $\Q(E'[2])$ contains $\Q(\sqrt \Delta)$ and $\Q(E'[4])$ contains $\Q(i)$, and since $\Q(E[2])$ is a $S_3$ extension of $E$, it follows that $\Delta$ is a square in $\Q(i)$, but not in $\Q$, i.e. $\Delta(E)\in -1\cdot (\Q^*)^2.$

By \cite[Lemma]{dd}, since $\Gal(\Q(E[4])/ \Q)\simeq S_3$, which is isomorphic to a subgroup of $H_{24}=\Z/3\Z \rtimes D_8$, it follows that $j(E)=-4t^3(t+8)$ for some $t\in \Q$. This concludes the proof of a).\\

Suppose now $E(K)$ has a point of order 8. It follows that $E(L)[8]\simeq \Z/4\Z \oplus \Z /8 \Z$. Let $G:=\Gal(L/M)$ and take the short exact sequence

\begin{equation}
\label{ses3}
0\rightarrow E(L)[2] \rightarrow E(L)[8] \rightarrow E(L)[8]/E(L)[2] \rightarrow 0.
\end{equation}
It follows that
\begin{equation}
\label{ses4}
0 \rightarrow E(L)[2]^G \rightarrow E(L)[8]^G \rightarrow (E(L)[8]/E(L)[2])^G \rightarrow H^1(G,E(L)[2]).
\end{equation}
Note that $E(L)[8]/E(L)[2]\simeq \Z/2\Z \oplus \Z /4 \Z$ and since we have shown in the proof of Proposition \ref{2-stvaranje} that $\Z/2\Z \oplus \Z /4 \Z$ has a $G$-invariant line, it follows that
$(E(L)[8]/E(L)[2])^G\neq 0$. Now from the fact that $H^1(G,E(L)[2])=0$ it follows that $E(L)[8]^G\neq 0$, from which it follows $E(M)[2]\neq 0$, which is a contradiction.
\end{proof}

\begin{remark}
Note that there exist elliptic curves such that $\Q(E[2])=\Q(E[4])$ and
$$\Gal(\Q(E[2])/ \Q)\simeq \Gal(\Q(E[4])/ \Q)\simeq S_3.$$
The elliptic curve 1936D1 is such a curve.
\end{remark}

We can combine Propositions \ref{2-stvaranje} and \ref{2-stvaranje2} into the following corollary.

\begin{corollary}
\label{nema4}
Let $E/ \Q$ be an elliptic curve such that $E(\Q)$ has no points of order $4$. Then $E(K)$ has no $8$-torsion and has a point of order $4$ only if $E(\Q)[2]=0$, $M=\Q(i)$, $\Delta(E)\in -1\cdot (\Q^*)^2,$ and $j(E)=-4t^3(t+8)$ for some $t\in \Q$.
\end{corollary}
\begin{proof}
If $E(\Q )$ has a $2$-torsion point, the statement follows from Lemma \ref{2-syl}.

If $E(\Q )$ has no $2$-torsion, the statement follows from Propositions \ref{2-stvaranje} and \ref{2-stvaranje2}.
\end{proof}

The next step towards the proof of Theorem \ref{gltm} is to control the growth of the $3$-torsion, for which the following two propositions will be useful.

\begin{lemma}
\label{nova3tor}
If $E(M)$ does not have a point of order $3$, neither does $E(L)$.
\end{lemma}
\begin{proof}
Let $G=\Gal(L/M)$. Then $E(L)[3]$ is $\F_3$-linear representation of $G$. Thus if $E(L)[3]\neq0$, then $E(L)[3]^G=E(M)[3]\neq 0$ by \cite[Proposition 26, p.64.]{ser}.
\end{proof}

\begin{proposition}
\label{9izogenija}
Suppose $E(K)$ has a point of order $9$. Then $E/ \Q$ has an isogeny of degree $9$ or $2$ independent isogenies of degree $3$.
\end{proposition}
\begin{proof}
Suppose the opposite, that $E / \Q$ does not have an isogeny of degree 9 nor 2 independent isogenies of degree 3.  Let $\diam{\sigma}=\Gal(L/M)$, $\diam{\tau}=\Gal(L/K)$ and $P\in E(K)$ be a point of order 9. It is easy to see that $E(L)$ has a $9$-torsion point, and that $E(M)$ has a point of order 3 by Lemma \ref{nova3tor}. Also, $E(M)$ cannot have a point of order 9, since this would imply that $E / \Q$ has a 9-isogeny by Lemma \ref{spust}, which would contradict our assumption.

We examine how $\sigma$ acts on $P$. Let $E[9]=\diam{P,Q}$ and $P^\sigma=\alpha P +\beta Q$.

Suppose $\beta=0$. Then $\sigma$ fixes $\diam{P}$ and since $P$ is a $K$-rational point, $\tau$ also fixes $\diam{P}$. Since $\sigma$ and $\tau$ generate $\Gal(L/\Q)$, it follows that that $E / \Q$ has a 9-isogeny, contradicting our assumption.

Note that $P^\sigma \in E(L)$ so $(9-\alpha)P+P^\sigma =\beta Q \in E(L)$, and it follows that $\beta$ has to be $3$ or $6$, otherwise the full $9$-torsion would be defined over $L$, which would further imply that $L=\Q(\zeta_9 )$, which is impossible since $\Gal(L/\Q)\simeq \Z/3\Z$ or $S_3$ and $\Gal(\Q(\zeta_9) /\Q)\simeq \Z /6\Z$. Furthermore, it follows that $E(L)[9]\simeq \Z/3\Z \oplus \Z/9 \Z$ and from this $M=\Q(\sqrt{-3})$.

Let $E[3]=\diam{P',Q'}$. By Lemma \ref{nova3tor}, $E(M)$ has non-trivial $3$-torsion; suppose $P'\in E(M)$. Now $G=\Gal(L/M)$ acts on $\diam{Q'}$, and since $G$ is a group of order 3, it follows that $\diam{Q'}^G=\diam{Q'}$ and hence $E(M)$ has full $3$-torsion. Since $M=\Q(\sqrt{-3})$, it holds that $E(M)[3]=E(\Q)[3]\oplus E^{-3}(\Q)[3]$ by \eqref{zbrojkvad}, and since an elliptic curve over $\Q$ cannot have full $3$-torsion, it follows that $E(\Q)[3]\simeq \Z/3\Z$. Suppose without loss of generality that $P'\in E(\Q)$. Then $\diam{P'}$ is obviously a $\Gal(\overline \Q/ \Q)$-invariant subgroup and we will show that there exists another orthogonal to it. Let $\diam{\mu}=\Gal(M/\Q)$. Then
$$Q'^{\mu}=\alpha P'+\beta Q', $$
$$Q'^{\mu^2}=Q'=\alpha(1+\beta) P'+\beta^2 Q',$$
where $\alpha, \beta \in\{0,1,2\}$. It follows that $\beta=1$ or $2$. If $\beta =1$ then $\alpha=0$, which would imply that $E(\Q)$ has full $3$-torsion, which is impossible. Hence $\beta=2$ and $\diam{Q'+2\alpha P'}$ is a $\Gal(\overline \Q/ \Q)$-invariant subgroup orthogonal to $\diam{P'}$, which shows that $E$ has 2 independent 3-isogenies, giving a contradiction.
\end{proof}

We now move to the study of the $p$-torsion of $E(K)$ for $p>3$.

\begin{lemma}
\label{punatorzija}
For $p>3$ prime, $E(L)[p]=0 \text{ or } \Z / p \Z$.
\end{lemma}
\begin{proof}
This follows from the fact that $E(L)[p]\simeq \Z /p\Z \oplus \Z /p\Z$ would, by the existence of the Weil pairing, require the $p$-th cyclotomic field $\Q(\zeta_p)$ to be contained in $L$, which is impossible.
\end{proof}

Next we study for which $p$, $E(M)[p]=0$ implies $E(L)[p]=0$. We prove a more general statement, that applies beyond the case of cubic extensions.

\begin{lemma}
\label{novi}
Let $p,q$ be odd distinct primes, $F_2/F_1$ a Galois extension of number fields such that $\Gal(F_2/F_1)\simeq \Z/q\Z$, and $E/F_1$ an elliptic curve with no $p$-torsion over $F_1$. Then if $q$ does not divide $p-1$ and $\Q(\zeta_p)\not\subset F_2$, then $E(F_2)[p]=0$.
\end{lemma}
\begin{proof}
 Since if one point of order $p$ is defined over $F_2$, then so are all its multiples, it follows that either $p-1$ or $p^2-1$ points of order $p$ are defined over $F_2$, but not over $F_1$. But it is impossible that all $p^2-1$ appear because of $\Q(\zeta_p)\not\subset F_2$.

 Let $\diam{\sigma}=\Gal (F_2/F_1)$, $P$ be a point of order $p$ in $E(F_2)$ and then note that $P^\sigma \neq P$, and that $\Gal (F_2/F_1)$ acts on $\diam{P}$. We have a homomorphism
 $$\Z /q\Z\simeq \Gal(F_2/F_1)\rightarrow \Aut(\diam{P})\simeq \Aut(\Z/ p \Z )\simeq (\Z/p\Z)^\times,$$
 so either $\Gal(F_2/F_1)$ acts trivially on $P$ and therefore $P\in E(F_1)[p]$ or $q$ divides $p-1$.

\end{proof}

In this paper we will use only the special case $q=3$, $F_1=M$ and $F_2=L$ of Lemma \ref{novi}. We also need to show the non-existence of points of order $p^2$ in $E(L)$. Again, we prove a more general statement.

\begin{lemma}
\label{prosirenje}
Let $p$ be an odd prime number, $q$ a prime not dividing $p$, $F_2/F_1$ a Galois extension of number fields such that $\Gal(F_2/F_1)\simeq \Z/q\Z$, $E/F_1$ an elliptic curve, and suppose $E(F_1)\supset \Z/ p\Z$, $E(F_1)\not\supset \Z/ p^2\Z$ and $\zeta_p\not \in F_2$. Then $E(F_2)\not\supset \Z/ p^2\Z$.
\end{lemma}
\begin{proof}
Suppose $\Z/p^2\Z\subset E(F_2)$. By the assumption $\zeta_p\not \in F_2$ and the existence of the Weil pairing, it follows that $E(F_2)[p^2]\simeq \Z/p^2\Z$. Let $P\in E(F_1)$ be of order $p$ and $\diam{\sigma}=\Gal(F_2/F_1)$. Let $$S=\{Q\in E(F_2)| pQ=P\}.$$
The set $S$ has $p$ elements, on which $\Gal(F_2/F_1)$ acts. By the Orbit Stabilizer Theorem, the orbits under the action of $\Gal(F_2/F_1)$ have to have length $q$, since if a point in $S$ was left fixed, it would mean that it is $F_1$-rational. This implies that $S$ decomposes into orbits of $q$ elements each, which is a contradiction with our assumption that $q$ does not divide $p$.
\end{proof}

We will again use Lemma \ref{prosirenje} only in the special case $q=3$, $F_1=M$ and $F_2=L$. For $n$ coprime to 6, the existence of a point of order $n$ in $E(K)$ will imply the existence of an $n$-isogeny over $\Q$, as the following lemma shows.

\begin{lemma}
\label{izogenije}
Let $n$ be an odd integer not divisible by $3$ and suppose $E(K)$ has a point of order $n$. Then $E/ \Q$ has an isogeny of degree $n$.
\end{lemma}
\begin{proof}
First note that $E(L)$ has a point $P$ of order $n$. Let
$$\diam{\sigma}=\Gal (L/M),\ \diam{\tau}=\Gal (L/K),\ \diam{\sigma,\tau}=\Gal(L/ \Q).$$
As $P$ is $K$-rational, it follows that $P^\tau=P$. Let $E[n]=\diam{P,Q}$ and $P^\sigma=\alpha P+\beta Q \in E(L)$. We have $(n-\alpha)P+P^\sigma=\beta Q \in E(L)$. If $\beta\not\equiv 0 \pmod n$ then $\beta Q$ is a point of order $l|n$ not contained in $\diam{P}$, from which it follows that  that $E(L)$ has full $l$-torsion, which is impossible by Lemma \ref{punatorzija}. We conclude that $\beta \equiv 0 \pmod n$.

Hence
$$P^\mu=kP \text{ for all }\mu \in \Gal(L/ \Q ),$$
and since the action of $\Gal(\overline \Q /\Q)$ on $\diam{P}$ factors through $\Gal(L/\Q)$, it follows that
$$P^\mu=kP \text{ for all }\mu \in \Gal(\overline \Q / \Q ),$$
which means that $E/ \Q$ has an $n$-isogeny.
\end{proof}

In the special case when $K=L$, the conclusion of Lemma \ref{izogenije} follows when $n$ is a multiple of 3.

\begin{lemma}
\label{galizogenije}
Suppose $K=L$, i.e. $K/ \Q$ is a Galois extension. Let $n$ be an odd integer and suppose $E(K)$ has a point of order $n$. Then $E/ \Q$ has an isogeny of degree $n$.
\end{lemma}
\begin{proof}
The proof follows by a similar argument as in Lemma \ref{izogenije} and by using the fact that $\Q (\zeta_k)$ is not contained in $K$, for any divisor $k\geq 3$ of $n$.
\end{proof}

\section{Proof of Theorem \ref{gltm}}

We are now ready to prove Theorem \ref{gltm}, which we will do in a series of Lemmas and Propositions. Recall that if a point in $E(K)$ has prime order $p$, then $p\leq 13$ (see \cite{Par1, Par2}).

\begin{lemma}
\label{3-syl}
The $3$-Sylow subgroup of $E(K)$ is isomorphic to a subgroup of $\Z/ 9\Z$.
\end{lemma}
\begin{proof}
This follows by \cite[Theorem (4.1)]{mom} and by the existence of the Weil pairing.
\end{proof}

\begin{lemma}
\label{5tors}
The $5$-Sylow groups of $E(\Q)$ and $E(K)$ are equal.

\end{lemma}
\begin{proof}
First note that $E(L)$ cannot have full $5$-torsion by Lemma \ref{punatorzija}.

Suppose $E(\Q )[5]\simeq \Z /5\Z$. Then $E(M)\simeq \Z/5\Z$ and Lemma \ref{prosirenje} shows that $E(L)$ cannot have a point of order 25, and hence the $5$-Sylow groups of $E(\Q)$ and $E(L)$ (and hence also $E(K)$) are equal and all isomorphic to $\Z/5\Z$.

Suppose $E(\Q )[5]=0$. If $E(M)[5]=0$, then  $E(L)[5]=0$ (and hence $E(K)[5]=0$) by Lemma \ref{novi}. If $E(M)[5]\simeq \Z /5\Z$, then $0=E(\Q)[5]=E(M)[5] \cap E(K)[5]$, from which it follows that $E(K)[5]=0$ (note that $E(K)$ here cannot have a $5$-torsion point that is not in $E(M)$ because then $E(L)$ would have full $5$-torsion).
\end{proof}

\begin{lemma}
\label{11tors}
There are no points of order $11$ in $E(K)$.
\end{lemma}
\begin{proof}
By Lemma \ref{novi}, $E(L)$ has $11$-torsion only if $E(M)$ has $11$-torsion, which is never true, as rational elliptic curves cannot have $11$-torsion over $\Q$ or over a quadratic field, by Theorem \ref{kvkl}. Hence $E(K)$ has no $11$-torsion.
\end{proof}

\begin{lemma}
\label{elim}
$E(K)$ has no points of order $35,$ $49,$ $65,$ $91$ or $169$.
\end{lemma}
\begin{proof}
This follows by Lemma \ref{izogenije} and Theorem \ref{izogklas}.
\end{proof}

\begin{lemma}
\label{15tors}
There exists no rational elliptic curves with points of order $15$ or $16$ over a cubic field.
\end{lemma}
\begin{proof}
As $E(L)[5]=E(M)[5]$ and $E(M)[3]=0\Longrightarrow E(L)[3]=0$ by Lemmas \ref{novi} and \ref{nova3tor}, it follows that the only way for $E(L)$ to have $15$-torsion is for $E(M)$ to have $15$-torsion.

It follows that, by Theorem \ref{kvkl} c), $E$ is 50B1, 50B2, 50A3 or 450B4. By the division polynomial method, we find that none of these curves have points of order $15$ over any cubic field.\\

If $E(\Q)$ has a non-trivial $2$-Sylow group, then the $2$-Sylow subgroups of $E(\Q)$ and $E(K)$ are equal by Lemma \ref{2-syl} and hence $E(K)$ has no points of order $16$. If the $2$-Sylow subgroup of $E(\Q)$ is trivial, then by Corollary \ref{nema4}, $E(K)$ has no points of order $8$.
\end{proof}

\begin{proposition}
\label{21tors}
The elliptic curve $162B1$ has torsion isomorphic to $\Z/21 \Z$ over $\Q(\zeta_9)^+$. This is the unique pair $(E,K)$ of a rational elliptic curve $E$ and a cubic field such that $E(K)$ has a point of order $21$.
\end{proposition}
\begin{proof}
By Lemma \ref{nova3tor}, if $E(M)_{tors}= 0, \text{ or }\Z /7 \Z $, then $E(L)[3]=0$. Suppose now that $E(K)[21]\supset \Z/21 \Z$. By Lemma \ref{izogenije}, it follows that $E/\Q$ has an isogeny of degree 7, and since it has a $3$-torsion point, it also has a $21$-isogeny. There are 4 curves (up to $\overline \Q$-isomorphism) with a 21-isogeny \cite[p.78--80]{mod}. These are the curves in the 162B or 162C isogeny classes. Note that the 162B isogeny class is a $-3$ twist of the 162C class.

By the division polynomial method we find that the only twists of the curves in the 162B and 162C isogeny classes with non-trivial $3$-torsion are the curves 162C1, 162C3, 162B1 and 162B3. By the division polynomial method, we find that only 162B1 has a $7$-torsion point over a cubic field, the field $\Q(\zeta_9)^+$, which is generated by $x^3-3x^2+3$.
\end{proof}

Note that since 162B1 is the unique curve with $21$-torsion over any cubic field and since it has torsion exactly $\Z / 21 \Z$, this means that there exist no points of order $21n$ on rational elliptic curves over cubic fields, for any integer $n\geq 2$.

\begin{remark}
Note that in \cite[Remark 2]{naj2} we misstated that from the fact that
\begin{equation}X_0(21)(\Q (\zeta_9)^+)=X_0(21)(\Q)\label{greska}\end{equation}
(note that $K_7=\Q (\zeta_9)^+,$ using the notation of \cite{naj2}) one can conclude that there are no elliptic curves with $21$-torsion over $\Q (\zeta_9)^+$, while what should have been written is that from (\ref{greska}) we can determine whether $21$-torsion appears over $\Q (\zeta_9)^+$ by checking whether twists of rational elliptic curves with rational $21$-isogeny have $21$-torsion over $\Q (\zeta_9)^+$. The point of the remark, that if we could find out whether $Y_1(25)(\Q(\zeta_9)^+)=\emptyset$, then we could completely classify the possible torsion groups of elliptic curves over $\Q(\zeta_9)^+$, remains true.
\end{remark}

\begin{lemma}
\label{20-24-26-28-36-39tors}
$E(K)$ does not have points of order $20$, $24$, $26$, $28$, $36$ or $39$.
\end{lemma}
\begin{proof}
Suppose $E(K)$ has a $20$-torsion point. Then $E(M)$ has to have a $5$-torsion point by Lemma \ref{novi}, and by (\ref{zbrojkvad}) either $E^d(\Q)$ or $E(\Q)$ has a $5$-torsion point, where $M=\Q(\sqrt d)$. We can suppose without loss of generality that $E(\Q)$ has a $5$-torsion point, otherwise we proceed with the proof using $E^d$ instead of $E$.

Thus, by \cite[Table 3.]{kub}, $E/\Q$ has a model
\begin{equation}
\label{jed5}
E:y^2+(1-t)xy -ty=x^3-tx^2 \text{ for some } t\in \Q^*.
\end{equation}
If $E(\Q)$ had a $4$-torsion point, it would follow that it has a $20$-torsion point which is impossible. Thus by Corollary \ref{nema4}, it follows that $E(\Q)[2]=0$ and $\Delta(E)=-k^2$ for some $k\in \Q^*$.
Hence
$$-k^2=\Delta(E)=t^5(t^2-11t-1)$$
for some $t\in \Q^*$, which is equivalent to
$$X:y^2=t^3+11t^2-t$$
having rational points with $t\in \Q^*$. But $X(\Q)=\{0, (0,0)\}$.\\

If $E(K)$ had a $24$-torsion point, then Corollary \ref{nema4} would imply that $E(\Q)$ has a non-trivial $2$-Sylow group, and from this fact and Lemma \ref{2-syl} it follows that $E(\Q)$ has a point of order 8. By Lemma \ref{nova3tor}, $E(M)$ has $3$-torsion point and hence a $24$-torsion point which is in contradiction with Theorem \ref{kvkl}, since $M/\Q$ is quadratic.\\

Suppose $E(K)\supset \Z /26 \Z$. It follows by Lemma \ref{izogenije} that $E/ \Q$ has an isogeny of degree 13, and this implies that $E( \Q)[2]= 0$, since
by Theorem \ref{izogklas} there are no elliptic curves with 26-isogenies over $\Q$. Because $E$ has a 13-isogeny, but no $13$-torsion over $\Q$, it follows that $\Q(\diam{P})/\Q$ is Galois and non-trivial. Since this is contained in $K$ and $K$ is of prime degree, we must have $\Q(\diam{P})=K$ and therefore $K/\Q$ is Galois.


From the facts that $K/\Q$ is cubic and Galois, $E(\Q)[2]=0$ and $E(K)[2]\neq 0$, it follows that $E(K)$ has full $2$-torsion and hence $E(K)\supset \Z /2 \Z \oplus\Z /26 \Z$.

By \cite{klo,mes} an elliptic curve with a $13$-isogeny over $\Q$ has $j$-invariant
$$j=\frac{(t^2+5t+13)(t^4+7t^3+20t^2+19t+1)^3}{t}, \ t\in \Q^\times.$$
It follows that $E$ is a quadratic twist of an elliptic curve $E_0$ with discriminant
$$\Delta(E_0)=t(t^2+5t+13)^2(t^4+7t^3+20t^2+19t+1)^6(t^2+6t+13)^9$$
$$\times(t^6+10t^5+46t^4+108t^3+122t^2+38t-1)^6,$$
and thus $\Delta(E)=u^{12}\Delta(E_0)$ for some $u\in \Q^\times$. The curve $E$ gains full $2$-torsion over a cubic field only if $\Delta(E)$ is a square which happens only if
$$X:y^2=x(x^2+6x+13) \text{ for some }x,y \in \Q^\times$$
has solutions. But $X(\Q)=\{0, (0,0) \}$, and hence that is impossible.\\

Suppose $E(K)$ has a $28$-torsion point. It follows that $E/\Q$ has a 7-isogeny by Lemma \ref{izogenije}. If $E(\Q )$ had a $4$-torsion point, this would imply that $E/ \Q$ has a 28-isogeny, which is impossible by Theorem \ref{izogklas}. Thus $E(\Q)$ does not have a $4$-torsion point and by Corollary \ref{nema4} it follows that $K$ is not a Galois extension of $\Q$ and $\Delta(E)=-k^2$, for some $k\in \Q^*$.

Note that since $E$ cannot have 2 independent rational 7-isogenies (see \cite[Proposition III.2.1.]{kub}), it follows that the kernel of the rational 7-isogeny is equal to $E(K)[7]$. As the points in the kernel of the rational 7-isogeny are defined over a Galois extension of $\Q$, it follows that they must be defined already over $\Q$, since they cannot be $K$-rational but not $\Q$-rational (because $K/ \Q$ is not Galois).

Hence $E(\Q)$ has points of order 7, and by \cite[Table 3.]{kub}, it follows that $E$ is of the form
\begin{equation}
\label{jed7}
E:y^2 + (-t^2 + t + 1)xy + (-t^3 + t^2)y = x^3 +(-t^3 + t^2)x^2, \text{ for some } t\in \Q, t \not\in\{0,1 \}.
\end{equation}
and that
$$-k^2=\Delta(E)=t^7(t-1)^7(t^3-8t^2+5t-1),$$
which is equivalent to
$$X:y^2=t(t+1)(t^3+8t^2+5t+1),$$
having rational points such that $t\not\in \{0,-1\}$. But the Jacobian $J$ of $X$ has rank 0 over $\Q$ and it is an easy computation in Magma (using the \texttt{Chabauty0} function) to show that $$X(\Q)=\{\infty, (0,0),(-1,0)\},$$ and thus completing the proof that there are no rational elliptic curves with $28$-torsion over a cubic field.\\

Suppose $E(K)$ has a $36$-torsion point. This implies that, by Corollary \ref{nema4}, $E(\Q)$ has either a point of order $4$ or no $2$-torsion. Also, by Proposition \ref{9izogenija}, $E/ \Q$ has to have a $9$-isogeny or $2$ independent isogenies of degree $3$.

Suppose first that $E(\Q)$ has a $4$-torsion point. If $E/ \Q$ had a $9$-isogeny, this would imply, by Lemma \ref{izogzbrajanje} that $E/ \Q$ is isogenous to a rational curve which has a 36-isogeny, which is impossible by Theorem \ref{izogklas}. On the other hand, by \cite[Main Result 2.]{kub}, an elliptic curve with 2 independent $3$-isogenies cannot have a $4$-torsion point.

Suppose now that $E(\Q)$ has no $2$-torsion. We split this case into $2$ subcases: when $E(\Q)$ has a rational $9$-isogeny, and when $E$ has $2$ independent rational isogenies of degree $3$.

Suppose $E$ has a $9$-isogeny. Then, by \cite{kub, ing}, $E$ is a twist of an elliptic curve $E_0$ with $j$-invariant
\begin{equation}
\label{jed9}
j=\frac{t^{12} - 72t^9 + 1728t^6 - 13824t^3}{t^3 - 27},\ t\in \Q \backslash \{0,3 \}
\end{equation}
and
$$\Delta(E_0)=2^{12}3^6(t^3 - 27),$$
and since $E$ has to be a twist of $E_0$, it follows that $\Delta(E)=u^{12}\Delta(E_0)$, for some $u\in \Q^\times$. By Corollary \ref{nema4}, the curve $E$ gains a $4$-torsion point over a cubic field only if $-y^2=\Delta(E)\neq 0$ has solutions, which is equivalent to
$$X:y^2=t^3+27,\ t\in \Q \backslash \{0,-3 \}, y\in \Q$$
having a solution. But $X(\Q)=\{0, (-3,0) \}$.

Suppose now that $E$ has 2 independent rational 3-isogenies and that $E$ gains a point of order 9 over $K$. By the same type of argument as in the proof of Proposition \ref{9izogenija}, it follows that $M=\Q(\sqrt{-3})$. Since $E$ has no rational $4$-torsion, but has a $4$-torsion point over $K$, by Corollary \ref{nema4}, it follows that $M=\Q(i)$, which is a contradiction.\\

If $E(K)$ had a $39$-torsion point, this would imply that $E(M)$ has a $3$-torsion point by Lemma \ref{nova3tor}, from which it would follow, by Lemma \ref{spust}, that $E/ \Q$ has a 3-isogeny. Also, $E/\Q$ would have a $13$-isogeny by Lemma \ref{izogenije}. But this would imply that there is an elliptic curve over $\Q$ with a 39-isogeny by Lemma \ref{izogzbrajanje}, which is impossible by Theorem \ref{izogklas}.
\end{proof}

\begin{lemma}
\label{neciklickators}
$E(K)$ cannot have subgroups isomorphic to $\Z /2\Z \oplus \Z /10 \Z$, $\Z /2\Z \oplus \Z /12 \Z$ or $\Z /2\Z \oplus \Z /18 \Z$.
\end{lemma}
\begin{proof}
Suppose $E(K)\supset \Z /2 \Z \oplus \Z /10 \Z$. By Lemma \ref{novi}, this implies that $E(\Q )$ has a $5$-torsion point. This implies, by \cite[Table 3.]{kub}, that $E$ has a model as in (\ref{jed5}). Since $E(\Q)$ cannot contain $\Z /2\Z \oplus \Z /10 \Z$, either $E(\Q)[2]=0$ or $\Z/2\Z$. But if $E(\Q)[2]=\Z / 2\Z$, then $E$ would gain full $2$-torsion over a quadratic instead of a cubic field. We conclude that $E(\Q)[2]=0$, and that $E$ gains full $2$-torsion over the cubic field $K$. This happens if and only if the discriminant $\Delta(E)$ is a square in $\Q^\times$, i.e. the equation
$$\Delta(E)=y^2=t^7 - 11t^6 - t^5, \text{ for some } y,t\in \Q^\times$$
has solutions. Dividing out by $t^4$ and by change of variables we get
$$X:A^2=t^3-11t^2-t, \text{ for some }A,t\in \Q^\times.$$
The curve $X$ is an elliptic curve and $X(\Q )\simeq \Z / 2\Z$, where the rational points are $0$ and $(0,0)$. Thus $E(K)\not\supset \Z /2 \Z \oplus \Z /10 \Z$.\\

Suppose $E(K)\supset \Z /2 \Z \oplus \Z /12 \Z$. Also, $E(\Q )$ has to either have a $4$-torsion point or no $2$-torsion by Lemma \ref{2-syl} and Corollary \ref{nema4}.

If $E(\Q)$ had a $4$-torsion point, then by Lemma \ref{2-syl}, $E(\Q)\supset \Z /2 \Z \oplus \Z /4 \Z$. We see that $E(M)$ has a $3$-torsion point by Lemma \ref{nova3tor} and hence by Lemma \ref{spust} it follows that $E/ \Q$ has a 3-isogeny and moreover a $4$-torsion point is in the kernel of a rational 12-isogeny. It follows that $E/ \Q$ has a 12-isogeny and an independent 2-isogeny and now from Lemma \ref{izogzbrajanje} it follows that there exists an elliptic curve over with a 24-isogeny over $\Q$, which is in contradiction with Theorem \ref{izogklas}.

If $E(\Q)$ had trivial $2$-torsion, then $K/ \Q$ would have to be a Galois extension for $E(K)$ to have full $2$-torsion. But then by Corollary \ref{nema4}, $E(K)$ cannot have points of order $4$, which is a contradiction.\\

Suppose $E(K)\supset \Z /2 \Z \oplus \Z /18 \Z$. By Lemma \ref{9izogenija}, $E/ \Q$ has either a 9-isogeny or two independent isogenies of degree 3.

Suppose now that $E(\Q)$ has a $2$-torsion point. Then it follows, by Lemma \ref{2-syl}, that $E(\Q)$ has full $2$-torsion. But an elliptic curve with full $2$-torsion cannot have a 9-isogeny \cite[Table 2.]{kub} nor 2 independent 3-isogenies \cite[Proposition III.2.3]{kub}. Thus it follows that $E(\Q)[2]=0$ and from this that $K/\Q$ is a Galois extension. Now it follows, by Lemma \ref{galizogenije}, that $E/ \Q$ in fact has a 9-isogeny.

By \cite[Appendix]{ing}, since $E$ has a 9-isogeny, it is a twist of an elliptic curve $E_0$ with $j$-invariant as given in (\ref{jed9})
and $\Delta(E_0)=2^{12}3^6(t^3 - 27),$
and since $E$ has to be a twist of $E_0$, it follows that $\Delta(E)=u^{12}\Delta(E_0)$, for some $u\in \Q^\times$. The curve $E$ gains full $2$-torsion over a cubic field only if $\Delta(E)$ is a square, which is equivalent to
$$X:y^2=t^3-27,\ t\in \Q \backslash \{0,3 \}, y\in \Q$$
having a solution. But $X(\Q)=\{0, (3,0)\}$, so there exist no such curves.
\end{proof}

It is shown in Lemmas \ref{2-syl} and Proposition \ref{2-stvaranje2} that $2$-Sylow subgroup of $E(K)_{tors}$ is contained in $\Z/2\Z \oplus \Z/8\Z$, in Lemma \ref{3-syl} that the largest power of $3$ that can divide $|E(K)_{tors}|$ is $9$, and in Lemmas \ref{5tors} and \ref{elim} that no powers of $5,$ $7$ and $13$ (apart from $5,$ $7$ and $13$ themselves) divide $|E(K)_{tors}|$. By \cite{Par1, Par2} and Lemma \ref{11tors} the aforementioned primes are the only ones that can divide $|E(K)_{tors}|$.

Now Lemma \ref{elim}, \ref{15tors} and  \ref{20-24-26-28-36-39tors} and Proposition \ref{21tors} show that $|E(K)_{tors}|$ is divisible by more than one element from the set $\{3,5,7,13\}$ only for the pair $(E,K)$ from Proposition \ref{21tors}. The possible combinations of $2$-Sylow subgroups and $p$-Sylow subgroups for $p=3,5,7,13$ are dealt with in Lemmas \ref{20-24-26-28-36-39tors} and \ref{neciklickators} and Proposition \ref{14besk}.

This completes the proof that the groups that appear as torsion groups of rational elliptic curves over cubic fields are contained in the list (\ref{lista}). Note first that by \cite[Lemma 3.2 a)]{jks}, all of the groups from the list (\ref{rac}) appear infinitely often and the group $\Z /21 \Z$ has already been dealt with in Proposition \ref{21tors}.

Any elliptic curve $E/\Q$ with torsion isomorphic to $\Z /9\Z$ over $\Q$ gains a $2$-torsion point over a cubic field $K$ defined by the cubic polynomial $f(x)$, when $E$ is written in short Weierstrass form $E:y^2=f(x)$. Then by Lemmas \ref{20-24-26-28-36-39tors} and \ref{neciklickators} it follows that $E(K)_{tors}$ does not contain $\Z/36\Z$ or $\Z/2\Z \oplus \Z/18\Z$, respectively, and therefore it follows that $E(K)_{tors}\simeq \Z /18 \Z$.

It remains to prove, for each of the groups $T=\Z /13 \Z$, $\Z /14 \Z$ and $\Z /2 \Z \oplus \Z /14 \Z$, that there exist infinitely many elliptic curves $E$ and cubic fields $K$ such that $E(K)_{tors}\simeq T$.

We will deal with the groups $\Z /14 \Z$ and $\Z /2 \Z \oplus \Z /14 \Z$ simultaneously in the following proposition.

\begin{proposition}
\label{14besk}
There exists infinitely many elliptic curves $E/ \Q$ such that there exists a cubic field $K$ over which $E(K)_{tors} \simeq \Z /14 \Z$ and there exists infinitely many elliptic curves $E/ \Q$ such that there exists a cubic field $K$ over which $E(K)_{tors} \simeq \Z /2 \Z \oplus \Z /14 \Z$.
\end{proposition}
\begin{proof}
Let $E/ \Q$ be an elliptic curve such that $E(\Q)_{tors} \simeq \Z /7\Z$. By, \cite[Table 3.]{kub}, this curve has the model as given in (\ref{jed7}).
If $E$ is written in short Weierstrass form $y^2=f(x)$, then $E$ gains a point of order 2 over the cubic field $K$ generated by $f$. We will show that $E(K)_{tors} \simeq \Z/14\Z$. By Lemma \ref{20-24-26-28-36-39tors}, $E(K)$ cannot have a point of order 28, so it only remains to show that $E(K)_{tors} \not\simeq \Z /2 \Z \oplus \Z /14 \Z$

If $E$ gained full $2$-torsion over $K$, this would imply that
$$\Delta(E)=t^7(t-1)^7(t^3-8t^2+5t-1)$$
is a square in $\Q$, which is equivalent to
$$X:y^2=t(t-1)(t^3-8t^2+5t-1),$$
having rational points such that $t\not\in \{0,1\}$. But the Jacobian $J$ of $X$ has rank 0 over $\Q$ and it is an easy computation in Magma to show that $$X(\Q)=\{\infty, (0,0),(1,0)\},$$ proving the claim.

To prove that there exist infinitely many curves $E/ \Q$ such that for each curve there exists a cubic field $K$ such that $E(K)_{tors}  \simeq \Z /2 \Z \oplus\Z/14\Z$, we note that every elliptic curve from the infinite family of elliptic curves having torsion $\Z /2 \Z \oplus\Z/14\Z$ over a cubic field from \cite[Theorem 4.2]{jkl} has rational $j$-invariant and the cubic field over which it has $\Z /2 \Z \oplus\Z/14\Z$ torsion has discriminant $(t^6 + 4t^5 + 13t^4 - 40t^3 + 19t^2 + 36t + 31)^2$, which implies that it is Galois. These are lengthy but completely straightforward calculations, and hence we leave them out.

The fact that these curves have rational $j$-invariant does not yet prove that the curves are rational, but just that they are quadratic twists by $\delta\in O_K$ of some rational elliptic curve. We need to prove that in fact $\delta$ is rational.

Let $E_1$ be one of the curves from the family \cite[Theorem 4.2]{jkl} and let $E$ be a rational elliptic curve with the same $j$-invariant and denote $\diam{\sigma}=\Gal(K/ \Q)$. As already noted $E_1=E^\delta$. Let $P\in E^\delta(K)$ be a point of order 7. It follows that $P^\sigma$ is a point of order 7 in $E^{\sigma(\delta)}$ and that $P^{\sigma^2}$ is a point of order 7 in $E^{\sigma^2(\delta)}$.

Now we will show that in a family of quadratic twists over a cubic field $K$ there can be only one elliptic curve with a point of order 7. Suppose that $E_2/K$ and $E_2^d/K$ are quadratic twists by a $d\in K^\times$ which are not $K$-isomorphic, and that both have a point of order 7. Then it follows that
$$E_2(K(\sqrt d))[7] \simeq E_2(K)[7] \oplus E_2^d(K)[7] \supset \Z /7 \Z \oplus \Z /7 \Z,$$
or in other words $E$ has full $7$-torsion over $K(\sqrt d)$. Since $K(\sqrt d)$ has to contain $\zeta_7$ it follows that $K(\sqrt d)=\Q(\zeta_7)$. But elliptic curves over $\Q(\zeta_7)$ cannot have full $7$-torsion (see \cite[Theorem.]{ms}).

Thus it follows that $E^\delta$, $E^{\sigma(\delta)}$ and $E^{\sigma^2(\delta)}$ are all $K$-isomorphic which means that $E^\delta $ is a $\Q$-curve. It is known \cite[Section 2.]{elk} that $\Q$-curves are either rational or defined over a $(2,\ldots,2)$ extension of $\Q$. Hence $E^\delta$ is defined over $\Q$, completing the proof.
\end{proof}

\begin{remark}
\label{qcurve}
In the proof of Proposition \ref{14besk}, once we have proven that $E^\delta$, $E^{\sigma(\delta)}$ and $E^{\sigma^2(\delta)}$ are all $K$-isomorphic, an alternative way of proving that $E^\delta$ is rational, without using $\Q$-curves can be done in the following way.

We can see that $E^\delta$, $E^{\sigma(\delta)}$ and $E^{\sigma^2(\delta)}$ are $K$-isomorphic if and only if
$$ \delta\sigma(\delta), \sigma(\delta)\sigma^2(\delta) \text{ and } \delta\sigma^2(\delta)$$
are all squares in $K$.
But since
$$N_{K/ \Q}(\delta)=\delta\sigma(\delta)\sigma^2(\delta)=k\in \Q,$$
it follows that
$$\delta=ka^2 \text{ for some }k \in \Q \text{ and }a \in K,$$
or in other words that $E_1$ is a rational twist of $E/\Q$, and hence can be defined over $\Q$.
\end{remark}

Note that it is not hard to prove that there exist rational elliptic curves with non-trivial $13$-torsion over cubic fields; a short search in Cremona's tables shows that 147B1 is such a curve. The hard part is proving that there are infinitely many such curves. In fact, 147B1 is the only curve with this property that we found in our (short) search.

Let $\{\pm 1\} \leq \Delta\leq (\Z / N\Z )^\times$ and define the congruence subgroup
$$\Gamma_\Delta=\left\{ \begin{pmatrix} a & b \\ c & d\end{pmatrix}\in \SL_2(\Z) \lvert a \text{ mod } N\in \Delta, N|c \right\}.$$

For $N$ prime, the modular curve $Y_\Delta(N)$ corresponding to $\Delta$ has the following moduli space interpretation: a $F$-rational point on $Y_\Delta(N)$ corresponds to an isomorphism class of pairs $(E,\diam{P})$ of an elliptic curve $E/F$ and a subgroup $\diam{P} \in E(\overline \Q)$ of order $N$ such that every $\sigma \in \Gal(\overline \Q /F)$ acts on $\diam{P}$ as multiplication by some $\alpha(\sigma) \in \Delta$. Let $X_\Delta(N)$ be the compactification of $Y_\Delta(N)$. Note that if $\Delta=(\Z / N\Z )^\times$ then $X_\Delta(N)=X_0(N)$ and if $\Delta=\{\pm 1\}$ then $X_\Delta(N)=X_1(N)$ ($X_{\{\pm 1\}}$ is then defined as the quotient of $\mathbb H$ by $\pm \Gamma_1(N)$, but since the action of $\pm 1$ is trivial on $\mathbb H$, it follows that $X_1(N)=X_{\{\pm 1\}}(N))$. All intermediate curves between $X_1(N)$ and $X_0(N)$ are of the form $X_\Delta(N)$, for some $\Delta$.

We now prove that there are in fact infinitely many rational elliptic curves with non-trivial $13$-torsion over some cubic field.

\begin{proposition}
There exists infinitely many elliptic curves $E/ \Q$ such that there exists a cubic field $K$ such that $E(K)$ has a $13$-torsion point.
\end{proposition}
\begin{proof}
Let $\Delta=\{ \pm 1, \pm 3, \pm 4\}\subset (\Z / 13 \Z)^\times$. Then by \cite[Theorem 1.1. and Table 1.]{jk}, $X_\Delta(13)$ has genus 0. The existence of a $\Q$- rational cusp on $X_\Delta(13)$ now immediately shows that $X_\Delta(13)(\Q)$ (and hence $Y_\Delta(13)(\Q)$) has infinitely many points.

Now let $(E,\diam{R})$, where $E/ \Q$ and $\diam{R}$ is a $13$-cycle of $E$, be a point on $X_\Delta(13)(\Q)$. If every $\sigma\in \Gal(\overline \Q / \Q)$ acts on $\diam{R}$ by multiplication by an element of $(\Z / 13 \Z)^\times$ of order 3, then it would follow that $R$ is defined over a cubic field and we are done.

Suppose the opposite, that $\sigma$ acts on $\diam{R}$ by multiplication as an element of $(\Z / 13 \Z)^\times$ of order 6. It can be seen, however that when $E$ is written in short Weierstrass form,  $\sigma$ actually permutes the three $x$-coordinates of $\pm R, \pm 3R$ and $\pm 4R$ and since $x(T)=x(-T)$ for any $T\in E(\overline \Q)$, this implies that the $x$ coordinates of the points in $\diam{R}$ are defined over a cubic field $K$. Let $F\supset K$ be the field of definition of $R$. If $F=K$ we are done so suppose $F=K(\sqrt \delta)$, for some $\delta \in K$. Then $E(F)$ has a point of order 13 and since
$$E(F)[13]=E(K)[13]+E^\delta(K)[13],$$
it follows that either $E$ or $E^\delta$ have a $K$-rational point of order 13. If $E(K)$ has a point of order $13$ we are done. Suppose that $E^\delta(K)$ has a point of order 13. Since $F=\Q(\diam{R})$, the field $F$ is a Galois extension of $\Q$ with $\Gal(F/\Q)\simeq \Z /6 \Z$ and it follows that $K$ is Galois over $\Q$ with $\Gal(K/\Q)\simeq \Z/3\Z$. Let $\diam{\tau}= \Gal(K/\Q)$.

Using the same argument as in the proof of Proposition \ref{14besk}, one can prove that $E^\delta$, $E^{\tau(\delta)}$ and $E^{\tau^2(\delta)}$ all have to be $K$-isomorphic and hence it follows that $E$ is a $\Q$-curve and it follows that $E^\delta$ has to be defined over $\Q$.

Thus for every rational elliptic curve $E$ represented by a point on $X_\Delta(13)$, there exists a rational twist $E'$ such that there exists a cubic field $K$ with the property that $E'(K)$ has a point of order 13.
\end{proof}

\section{Sporadic points on $X_1(n)$}

As we have seen in Proposition \ref{21tors}, there exists a sporadic point of degree 3 on $X_1(21)$, which is a curve of gonality 4 (the gonality of $X_1(21)$ can be deduced from \cite[Theorem 2.3.]{jks} and \cite[Theorem 0.1.]{jk2}). This point was essentially constructed by starting with an elliptic curve $E/ \Q$ with a $21$-isogeny and then using the division polynomial method to determine the minimal degree of a field $K$ over which the points in the $\Gal(\overline \Q / \Q)$-invariant subgroup of order 21 of some twist of $E$ becomes $K$-rational.

It is a natural question to ask whether the same procedure can be used to find other sporadic points by starting with other rational elliptic curves with isogenies. We have tried this and this method gives us (only) a degree 6 point on $X_1(37)$; we describe the procedure used to find it below.

There are 2 families of twists of elliptic curves with $37$-isogenies. We start with the elliptic curve $E=1225H1$, take a short Weierstrass model
$$y^2 = x^3 - 10395x + 444150$$
of it and factor (over $\Q$) its $37$-division polynomial $\psi_{37}$ finding a degree 6 factor
$$f_6= x^6 - 3150x^5 + 796635x^4 - 75770100x^3 + 3111596775x^2 $$
$$- 44606598750x - 85333003875.$$
This implies that the $x$-coordinate of a point of order $37$ of $E$ is defined over a sextic field $F$ and since twisting does not change the roots of division polynomials, it follows that there exists an unique twist (over $F$) of $E$ such that it has a point of order $37$ over $F$. This can be found simply by finding over which quadratic extension $F(\sqrt{\delta})$ the points of order $37$ become defined, and then the quadratic twist we are looking for is $E^\delta$.
Let $w$ be a root of $f_6$. We compute that
$$\delta=w^3 - 10395w + 444150$$
and that $E^\delta$ indeed has a point of order $37$. Thus $E^\delta$, together with a point of order $37$ represents a sporadic point of degree 6 on $X_1(37)$, which has gonality 18. Note that this is the same curve which has been previously found by van Hoeij using computational methods \cite{hoe}.\\

\textbf{Acknowledgements}.

We thank K\k{e}stutis \v Cesnavi\v cius for finding a mistake in Proposition \ref{2-stvaranje2} in an earlier version and for helping simplify the proofs of Propositions \ref{2-stvaranje} and \ref{2-stvaranje2} and the anonymous referee for many suggestions that greatly improved this paper. We are also grateful to Andew Sutherland for providing useful computational data, and to Andrej Dujella, Enrique Gonz\'alez-Jim\'enez and Petra Tadi\'c for helpful comments.

\end{document}